\documentclass[12pt,a4paper,leqno]{article}
\usepackage{amsmath,amsthm,amssymb}

\newtheorem{theorem}{Theorem}

\newtheorem{corollary}[theorem]{Corollary}
\theoremstyle{definition}
\newtheorem{definition}[theorem]{Definition}
\newtheorem{example}[theorem]{Example}

\theoremstyle{remark}

\numberwithin{equation}{section}

\begin{document}

\title{Countable Normality.}         
\author{Maha Mohammed Saeed\footnote{ Maha Mohammed Saeed, King Abdulaziz University, Department of Mathematics, P.O.Box 80203, Jeddah 21589, Saudi Arabia, mamosamo@hotmail.com, mmmohammed@kau.edu.sa}    }     
\date{King Abdulaziz University, Department of Mathematics, Jeddah, Saudi Arabia.}
\maketitle

\begin{abstract}

 \rm  A. V. Arhangel'ski\u{i} introduced in 2012, when he was visiting the department of Mathematics at King Abduaziz University,  new weaker versions of normality, called \it $C$-normality, \rm and \it countable normality. \rm    
The purpose of this paper is to investigate countable normality property.   We prove that normality implies countable normality but the converse is not true in general. We present some examples to show relationships between countable normality and other weaker versions of normality such as  $C$-normality, $L$-noramlity,  and mild normality. We answer the following open problem of Arhangel'ski\u{i}:``\it Is there a Tychonoff space which is not $C$-normal ?". \rm
 Throughout this paper, we denote an ordered pair by $\langle x,y\rangle$, the set of positive integers by $\mathbb{N}$ and the set of  real numbers by $\mathbb{R}$. A $T_4$ space is a $T_1$ normal space, a  Tychonoff space is a $T_1$ completely regular space, and a $T_3$ space is a $T_1$ regular space. We do not assume $T_2$ in the definition of compactness and we do not assume regularity in the definition of Lindel\"{o}fness. For a subset $A$ of a space $X$, ${\rm int} A$ and $\overline{A}$ denote the interior and the closure of $A$, respectively. An ordinal $\gamma$ is the set of all ordinals $\alpha$ such that $\alpha<\gamma$. The first infinite ordinal is $\omega_0$, the first uncountable ordinal is $\omega_1$, and the successor cardinal of $\omega_1$ is $\omega_2$.


{\bf Subject code} :54D15, 54C10.

\end{abstract}

\section{Countable Normality.}

\begin{definition} (Arhangel'ski\u{i})\\
A  topological space $X$ is called \it $C$-normal \rm if there exist a normal space $Y$ and a bijective function $f:X\longrightarrow Y$ such that the restriction $f_{|_C}:C\longrightarrow f(C)$ is a homeomorphism for each compact subspace $C\subseteq X$\rm.  A topological space $X$ is called \it countably normal \rm if there exist a normal space $Y$ and a bijective function $f:X\longrightarrow Y$ such that the restriction $f_{|_C}:C\longrightarrow f(C)$ is a homeomorphism for each countable subspace $C\subseteq X$. \rm \end{definition}

\bigskip

$C$-normality was studied in \cite{KZ}. There is a Tychonoff $C$-normal space which is not countably normal, see the paragraph after Theorem \ref{T09} below.
A  topological space $X$ is called \it $L$-normal \rm \cite{KM} if there exist a normal space $Y$ and a bijective function $f:X\longrightarrow Y$ such that the restriction $f_{|_C}:C\longrightarrow f(C)$ is a homeomorphism for each Lindel\"{o}f subspace $C\subseteq X$.  \rm Since any countable space is Lindel\"{o}f, then any $L$-normal space is countably normal. The converse is not true in general. For example the countable complement topology on an uncountable set $X$ \cite{Steen} is countably normal because any countable subspace of $X$ is a discrete subspace, hence the discrete topology on $X$ and the identity function will witness the countable normality. But it is not $L$-normal because it is Lindel\"{o}f non-normal space. 

If $X$ is countable, then $X$ may not be countably normal, for example, $(\,\mathbb{Q},${\LARGE$\tau$}$_p\,)$, where {\LARGE$\tau$}$_p$ is the particular point topology and $\mathbb{Q}$ is the rationales, $p\in\mathbb{Q}$ \cite{Steen}, because no countable non-normal space is countably normal. It is also obvious that any normal space is countably normal, just by taking $X=Y$ and $f$ to be the identity function. The converse is not true in general. Here is an example of a Tychonoff countably normal space which is neither normal nor locally compact. 

\bigskip

\begin{example}\label{E02}
We modify the Dieudonn\'{e} Plank \cite{Steen} to define a new topological space. Let $$X=((\omega_2+1)\times (\omega_0+1))\setminus \{\langle\omega_2,\omega_0\rangle\}.$$ Write $X=A\cup B\cup N$, where $A=\{\langle\omega_2,n\rangle:n<\omega_0\}$, $B=\{\langle\alpha,\omega_0\rangle:\alpha<\omega_2\}$, and $N=\{\langle\alpha,n\rangle:\alpha<\omega_2$ and $n<\omega_0\}$. The topology {\Large$\tau$} on $X$ is generated by the following neighborhood system: For each $\langle\alpha,n\rangle\in N$, let ${\mathcal B}(\langle\alpha,n\rangle)=\{\{\langle\alpha,n\rangle\}\}$. For each $\langle \omega_2,n\rangle\in A$, let ${\mathcal B}(\langle \omega_2,n\rangle)=\{V_\alpha(n)=(\alpha,\omega_2]\times\{n\}:\alpha<\omega_2\}$. For each $\langle\alpha,\omega_0\rangle\in B$, let ${\mathcal B}(\langle \alpha,\omega_0\rangle)=\{V_n(\alpha)=\{\alpha\}\times (n,\omega_0]:n<\omega_0\}$. Then $X$ is Tychonoff non-normal space which is neither locally compact nor locally Lindel\"{o}f as any basic open neighborhood of any element in $A$ is not Lindel\"{o}f, hence not countable \cite{KM}. 

Now, define $Y=X=A\cup B\cup N$. Generate a topology {\Large$\tau$}$^\prime$ on $Y$ by the following neighborhood system: Elements of $B\cup N$ have the same local base as in $X$. For each $\langle \omega_2,n\rangle\in A$, let ${\mathcal B}(\langle \omega_2,n\rangle)=\{\{\langle \omega_2,n\rangle\}\}$. Then $Y$ is $T_4$ space because it is Hausdorff paracompact. 
In \cite{KM}, it was proved that $X$ is $L$-normal. Therefore, the modified Dieudonn\'{e} Plank $X$ is countably normal.$\rule{.1in}{.1in}$
\end{example}

\bigskip

\begin{theorem}
countable normality is a topological property.\end{theorem}

\begin{proof} Let $X$ be any countably normal space and $X\cong Z$. Let $Y$ be a normal space and $f:X\longrightarrow Y$ be a bijective such that $f_{|_C}:C\longrightarrow f(C)$ is a homeomorphism for each countable subspace $C$ of $X$. Let $g:Z\longrightarrow X$ be a homeomorphism. Then  
$f\circ g:Z\longrightarrow Y$ satisfies all requirements.\end{proof}

\bigskip

\begin{theorem}
countable normality is an additive property.\end{theorem}
\begin{proof} Let $X_\alpha$ be a countably normal space for each $\alpha\in\Lambda$. We show that their sum $\oplus_{\alpha\in\Lambda}X_\alpha$ is countably normal. For each $\alpha\in\Lambda$, pick a normal space $Y_\alpha$ and a bijective function $f_\alpha:X_\alpha \longrightarrow Y_\alpha$ such that $f_{\alpha_{|_{C_\alpha}}}:C_\alpha\longrightarrow f_\alpha(C_\alpha)$ is a homeomorphism for each countable subspace $C_\alpha$ of $X_\alpha$. Since $Y_\alpha$ is normal for each $\alpha\in\Lambda$, then the sum $\oplus_{\alpha\in\Lambda}Y_\alpha$ is normal, [3, 2.2.7]. Consider the function sum, see [3, 2.2.E], $\oplus_{\alpha\in\Lambda}f_\alpha:\oplus_{\alpha\in\Lambda}X_\alpha\longrightarrow\oplus_{\alpha\in\Lambda}Y_\alpha$ defined by $\oplus_{\alpha\in\Lambda}f_\alpha(x)=f_\beta(x)$ if $x\in X_\beta, \beta\in\Lambda$. Now, a subspace $C\subseteq \oplus_{\alpha\in\Lambda}X_\alpha$ is countable if and only if the set $\Lambda_0=\{\alpha\in\Lambda : C\cap X_\alpha\not=\emptyset\}$ is countable and $C\cap X_\alpha$ is countable for each $\alpha\in \Lambda_0$. If $C\subseteq \oplus_{\alpha\in\Lambda}X_\alpha$ is countable, then $(\oplus_{\alpha\in\Lambda}f_\alpha)_{|_C}$ is a homeomorphism because $f_{\alpha_{|_{C\cap X_\alpha}}}$ is a homeomorphism for each $\alpha\in\Lambda_0$.\end{proof}

\bigskip

A function $f:X\longrightarrow Y$ witnessing the countable normality of $X$ need not be continuous as seen in Example \ref{E02}. But it will be if $X$ is of countable tightness. Similar proof as in \cite{KM} will prove the next two theorems.

\bigskip

\begin{theorem}
If $X$ is countably normal and of countable tightness and $f:X\longrightarrow Y$ is a witness of the countable normality of $X$, then $f$ is continuous.
\end{theorem}


\bigskip

\begin{theorem}\label{T09}
If $X$ is $T_3$ separable countably normal and of countable tightness, then $X$ is normal \rm ($T_4$). \end{theorem}





We conclude from the above theorem that the Niemytzki plane \cite{Steen} and  a Mr\'{o}wka space $\Psi({\mathcal A})$, where ${\mathcal A}\subset [\omega_0]^{\omega_0}$ is mad \cite{Dou}, are examples of Tychonoff spaces which are not countably normal. Observe that such a Mr\'{o}wka space is an example of a $C$-normal space, being Hausdorff locally compact \cite{KZ}, which is not countably normal.
Countable normality is not multiplicative because, for example, the Sorgenfrey line is $T_4$ but its square is Tychonoff separable first countable space which is not countably normal  because it is not normal. Also, countable normality is not hereditary, take any compactification of the Sorgenfrey line square. We still do not know if countable normality is hereditary with respect to closed subspaces.

\bigskip

Recall that a \it Dowker space \rm is a $T_4$ space whose product with $I$, $I=[0,1]$ with its usual metric, is not normal.
M. E. Rudin used the existence of a Suslin line to obtain a Dowker space which is hereditarily separable and first countable \cite{Rudin1}. Using CH, I. Juh\'{a}sz, K. Kunen, and M. E. Rudin constructed a first countable hereditarily separable real compact Dowker space \cite{JKR}. Weiss constructed a first countable separable locally compact Dowker space whose existence is consistent with MA + $\neg$ CH  \cite{Wei}. By Theorem 1.6, such spaces are consistent examples of Dowker spaces whose product with $I$ are not countably normal.

\bigskip

Since any countable space is separable and of countable tightness and any $T_3$ second countable space is metrizable [3, 4.2.9], we conclude the following theorem.

\bigskip

\begin{theorem}
Every $T_1$ countable countably normal second countable space is metrizable. \end{theorem}

\bigskip

On 2012, Arhangel'ski\u{i} asked the following problem, see \cite{KZ}:``\it Is there a Tychonoff space which is not $C$-normal ?". \rm The answer is positive. Here is an example of a Tychonoff space which is not $C$-normal.

\bigskip

\begin{example}
Let $G=D^{\omega_1}$, where $D=\{0, 1\}$ with the discrete topology. Let $H$ be the subspace of $G$ consisting of all points of $G$ with at most countably many non-zero coordinates. Put $X=G\times H$. Raushan Buzyakova proved that $X$ cannot be mapped onto a normal space $Y$ by a bijective continuous function \cite{Bu}. \rm

{\bf Claim:} $H$ is Fr\'{e}chet.

Proof of Claim:
Let $A\subseteq H$ be arbitrary and pick any $x\in\overline{A}$. Let $E\subset\omega_1$ be the set of all coordinates $\alpha$ such that $x_{|_\alpha}=1$. By the definition of $H$, $E$ is countable. Write $E=\{\alpha_n : n<\omega_0\}$. Consider the basic open neighborhood $$U=(\prod_{\alpha\in\omega_1\setminus\{\alpha_0\}}D)\times \{1\})\bigcap H$$ of $x$. We have $U\cap A\not=\emptyset$. Pick $x_0\in U\cap A$. Thus $x_{0_{|_{\alpha_0}}}=1$. For $n>0$, consider the open neighborhood $$U=((\prod_{\alpha\in\omega_1\setminus\{\alpha_0, ... , \alpha_{n}\}}D)\times (\prod_{\{\alpha_0, ... , \alpha_{n}\}}\{1\}))\bigcap H$$ of $x$. We have $U\cap A\not=\emptyset$. Pick $x_n\in U\cap A$. Thus $x_{n_{|_{\alpha_i}}}=1$ for all $i\in\{0, 1, ..., n\}$. Now, $x_n\in A$ for each $n\in\omega_0$ and the sequence $x_n$ converges to $x$ because if $V$ is any basic open neighborhood of $x$, then $$V=((\prod_{\alpha\in\omega_1\setminus F}D)\times(\prod_{\alpha\in F_1}\{1\})\times (\prod_{\alpha\in F_2}\{0\}))\cap H$$ where $F=F_1\cup F_2$ is finite with $F_1\cap F_2=\emptyset$. If $E\cap F=\emptyset$, then $x_n\in V$ for all $n<\omega_0$. If $E\cap F\not=\emptyset$, let $k={\rm max}\{i\in\omega_0 : \alpha_i\in E\cap F\}$. Then $x_n\in V$ for each $n>k$. Therefore, $H$ is Fr\'{e}chet.       

Now, $X$ is a $k$-space because $H$ is a $k$-space being a Hausdorff Fr\'{e}chet space, [3, 3.3.20]. Since $G$ is $T_2$ compact, then $X=G\times H$ is a $k$-space, [3, 3.3.27].    
Now, It is clear that $X$ is Tychonoff. Suppose that $X$ is $C$-normal. Let $Y$ be a normal space and $f:X\longrightarrow Y$ be a bijection function such that $f_{|_C}:C\longrightarrow f(C)$ is a homeomorphism for each compact subset $C\subseteq X$. Since $X$ is a $k$-space, then $f$ is continuous, [3, 3.3.21], and this contradicts the Buzyakova's result. $\rule{.1in}{.1in}$ \end{example}

\bigskip

The above example shows that $C$-normality is not hereditary, just take any compactification of $X$. It also shows that not every $k$-space is $C$-normal. The Niemytzki plane and  a Mr\'{o}wka space $\Psi({\mathcal A})$, where ${\mathcal A}\subset [\omega_0]^{\omega_0}$ is mad, are examples of $C$-normal spaces which are not countably normal. They are $C$-normal being $T_2$ locally compact \cite{KZ}.

\bigskip

\section{Countable Normality and other Properties.}

Now, we study some relationships between countable normality and some other weaker versions of normality. First, we recall some definitions.

\begin{definition}
A subset $A$ of a space $X$ is called \it closed domain \rm \cite{Eng},  called also \it regularly closed, $\kappa$-closed, \rm if $A=\overline{{\rm int}A}$. A space $X$ is called \it mildly normal \rm \cite{SS},  called also \it $\kappa$-normal \rm \cite{Sh},  if for any two disjoint closed domains $A$ and $B$ of $X$ there exist two  disjoint open sets $U$ and $V$ of $X$ such that $A\subseteq U$ and $B\subseteq V$, see also  \cite{KS} and \cite{Kal1}. A space $X$ is called \it almost normal \rm \cite{SA} \cite{KH} if for any two disjoint closed subsets $A$ and $B$ of $X$ one of which is closed domain, there exist two disjoint open subsets $U$ and $V$ of $X$ such that $A\subseteq U$ and $B\subseteq V$. A subset $A$ of a space $X$ is called \it $\pi$-closed \rm  \cite{Kal2} if $A$ is a finite intersection of closed domains. A space $X$ is called \it $\pi$-normal \rm  \cite{Kal2} if for any two disjoint closed subsets $A$ and $B$ of $X$ one of which is $\pi$-closed, there exist two disjoint open subsets $U$ and $V$ of $X$ such that $A\subseteq U$ and $B\subseteq V$. A space $X$ is called \it quasi-normal \rm \cite{Za} if for any two disjoint $\pi$-closed subsets $A$ and $B$ of $X$, there exist two disjoint open subsets $U$ and $V$ of $X$ such that $A\subseteq U$ and $B\subseteq V$, see also \cite{Kal2}.\end{definition}

It is clear from the definitions that

\centerline{normal $\Longrightarrow$ $\pi$-normal $\Longrightarrow$ almost normal $\Longrightarrow$ mildly normal.}

\centerline{normal $\Longrightarrow$ $\pi$-normal $\Longrightarrow$ quasi-normal $\Longrightarrow$ mildly normal.}

Now, $(\,\mathbb{Q}\,,$ {\LARGE$\tau$}$_0\,)$, where {\LARGE$\tau$}$_0$ is the particular point topology is not countably normal. But, since the only $\pi$-closed sets are $\emptyset$ and $\mathbb{Q}$, then it is $\pi$-normal, hence quasi-normal, almost normal, and mildly normal. Here is an example of a countably normal space which is not mildly normal, hence neither quasi-normal, almost normal, nor $\pi$-normal.

\bigskip

\begin{example}
The modified Dieudonn\'{e} plank $X$ is countably normal but not mildly normal.\end{example}

\begin{proof} $X$ is not normal because $A$ and $B$ are closed disjoint subsets which cannot be separated by two disjoint open sets. Let $E=\{ n<\omega_0 : n \,\,\mbox{is even}\,\}$ and $O=\{ n<\omega_0 : n\,\, \mbox{is odd}\,\}$.  Let $K$ and $L$ be subsets of $\omega_2$ such that $K\cap L=\emptyset$, $K\cup L=\omega_2$, and the cofinality of $K$ and $L$ are $\omega_2$; for instance, let $K$ be the set of limit ordinals in $\omega_2$ and $L$ be the set of successor ordinals in $\omega_2$.  
Then $K\times E$ and $L\times O$ are both open being subsets of $N$. Define $C=\overline{K\times E}$ and $D=\overline{L\times O}$; then $C$ and $D$ are closed domains in $X$, being closures of open sets, and they are disjoint. Note that $C=\overline{K\times E}=(K\times E)\cup (K\times \{\omega_0\})\cup (\{\omega_2\}\times E)$ and $D=\overline{L\times O}=(L\times O)\cup (L\times \{\omega_0\})\cup (\{\omega_2\}\times O)$. Let $U\subseteq X$ be any open set such that $C\subseteq U$. For each $n\in E$ there exists an $\alpha_n <\omega_2$ such that $V_{\alpha_n}(n)\subseteq U$. Let $\beta=\sup\{\alpha_n : n\in E\}$; then $\beta <\omega_2$. Since $L$ is cofinal in $\omega_2$, then there exists  $\gamma\in L$ such that $\beta < \gamma$ and then any basic open set of $\langle\gamma , \omega_0\rangle\in D$ will meet $U$. Thus $C$ and $D$ cannot be separated.
Therefor, the modified Dieudonn\'{e} plank $X$ is countably normal but is not mildly normal.\end{proof}

\bigskip

\begin{theorem} If $X$ is $C$-normal space such that each countable subspace is contained in a compact subspace, then $X$ is countably normal.\end{theorem}

\begin{proof} Let $X$ be any $C$-normal space such that if $A$ is any countable subspace of $X$, then there exists a compact subspace $B$ such that $A\subseteq B$. Let $Y$ be a normal space and $f:X\longrightarrow Y$ be a bijective function such that $f_{|_C}:C\longrightarrow f(C)$ is a homeomorphism for each compact subspace $C$ of $X$. Now, let $A$ be any countable subspace of $X$. Pick a compact subspace $B$ of $X$ such that $A\subseteq B$, then $f_{|_B}:B\longrightarrow f(B)$ is a homeomorphism, hence $f_{|_A}:A\longrightarrow f(A)$ is a homeomorphism as $(f_{|_B})_{|_A}=f_{|_A}$.\end{proof}

\bigskip

\begin{corollary} If $X$ is a countably normal space such that each Lindel\"{o}f subspace is countable, then $X$ is $L$-normal.\end{corollary}

\bigskip

The next example is an application of the above theorem. It gives a Tychonoff countably normal space which is not almost normal.

\bigskip

\begin{example}
Consider the product space $\omega_1\times(\omega_1+1)$. It is not almost normal because the diagonal $\triangle=\{\langle \alpha , \alpha\rangle :\alpha <\omega_1\,\}$ is a closed domain which is disjoint from the closed set $K=\omega_1\times\{\omega_1\}$ and they cannot be separated by two disjoint open sets, see \cite{KH}. But $\omega_1\times(\omega_1+1)$ is $C$-normal being locally compact and local compactness implies $C$-normality, see \cite{KZ}. Now we characterize all countable subspaces of $\omega_1\times(\omega_1+1)$.

Claim: If a subspace $A$ of $\omega_1\times(\omega_1+1)$ is countable, then there is an $\alpha\in\omega_1$ such that $A\subseteq \alpha\times (\omega_1+1)$.

Proof of claim:  Let $A$ be any countable subspace of $\omega_1\times(\omega_1+1)$. If the condition does not hold, then for any $\alpha\in\omega_1$ there is $x\in A$ and $x$ does not belong to $\alpha\times(\omega_1+1)$. But $\omega_1$ in uncountable, then $A$ is uncountable, a contradiction. So the claim is proved.

We conclude from the above claim that each countable subspace $A$ of $\omega_1\times(\omega_1+1)$ is contained in a compact subspace $B$ of $\omega_1\times(\omega_1+1)$ of the form $B=(\alpha+1)\times (\omega_1+1)$, where $\alpha$ satisfies condition above. Thus, by the Theorem 2.3, $\omega_1\times(\omega_1+1)$ is countably normal. $\rule{.1in}{.1in}$\end{example}

\bigskip

We discovered that the Alexandroff Duplicate space of a countably normal space is countably normal. Recall that the Alexandroff Duplicate space $A(X)$ of a space $X$ is defined as follows: Let $X$ be any topological space. Let $X^\prime=X\times \{1\}$. Note that $X\cap X^\prime=\emptyset$. Let $A(X)=X\cup X^\prime$. For simplicity, for an element $x\in X$, we will denote the element $\langle x,1\rangle$ in $X^\prime$ by $x^\prime$ and for a subset $B\subseteq X$ let $B^\prime=\{x^\prime:x\in B\}=B\times\{1\}\subseteq X^\prime$. For each $x^\prime\in X^\prime$, let ${\mathcal B}(x^\prime)=\{\{x^\prime\}\}$. For each $x\in X$, let ${\mathcal B}(x)=\{U\cup(U^\prime\setminus\{x^\prime\}):U$ is open in $X$ with $x\in U\,\}$. Then ${\mathcal B}=\{{\mathcal B}(x):x\in X\}\cup\{{\mathcal B}(x^\prime):x^\prime\in X^\prime\}$ will generate a unique topology  on $A(X)$ such that $\mathcal B$ is its neighborhood system.  $A(X)$ with this topology is called the \it Alexandroff Duplicate of $X$ \rm \cite{Alex}. 

\bigskip

\begin{theorem}
If $X$ is countably normal, then its Alexandroff Duplicate $A(X)$ is also countably normal.\end{theorem}

\begin{proof} Let $X$ be any  countably normal space. Pick a normal space $Y$ and a bijective function $f:X\longrightarrow Y$ such that $f_{|_C}:C\longrightarrow f(C)$ is a homeomorphism for each countable subspace $C\subseteq X$. Consider the Alexandroff duplicate spaces $A(X)$ and $A(Y)$ of $X$ and $Y$ respectively. It is well-known that the Alexandroff Duplicate of a normal space is normal, hence  $A(Y)$ is also normal. Define $g:A(X)\longrightarrow A(Y)$ by $g(a)=f(a)$ if $a\in X$ and if $a\in X^\prime$, let $b$ be the unique element in $X$ such that $b^\prime =a$, then define $g(a)=(f(b))^\prime$. Then $g$ is a bijective function. Now, a subspace $C\subseteq A(X)$ is countable if and only if $C\cap X$ is countable in $X$ and $C\cap X^\prime$ is countable in $X^\prime$. Let $C\subseteq A(X)$ be any countable subspace. We show $g_{|_C}:C\longrightarrow g(C)$ is a homeomorphism. Let $a\in C$ be arbitrary. If $a\in C\cap X^\prime$, let $b\in X$ be the unique element such that $b^\prime=a$. For the smallest basic open neighborhood $\{(f(b))^\prime\}$ of the point $g(a)$ we have that $\{a\}$ is open in $C$ and $g(\{a\})\subseteq\{(f(b))^\prime\}$. If $a\in C\cap X$. Let $W$ be any open set in $Y$ such that $g(a)=f(a)\in W$. Consider $H=(W\cup(W^\prime\setminus\{(f(a))^\prime\}))\cap g(C)$ which is a basic open neighborhood of $f(a)$ in $g(C)$. Since $f_{|_{C\cap X}}:C\cap X\longrightarrow f(C\cap X)$ is a homeomorphism, then there exists an open set $U$ in $X$ with $a\in U$ and $f_{|_{C\cap X}}(U\cap C)\subseteq W\cap f(C\cap X)$. Now, $(U\cup(U^\prime\setminus\{a^\prime\}))\cap C=G$ is open in $C$ such that $a\in G$ and $g_{|_C}(G)\subseteq H$. Therefore, $g_{|_C}$ is continuous. Now, we show that $g_{|_C}$ is open. Let $K\cup(K^\prime\setminus\{k^\prime\})$, where $k\in K$ and $K$ is open in $X$, be any basic open set in $A(X)$, then $(K\cap C)\cup ((K^\prime\cap C)\setminus\{k^\prime\})$ is a basic open set in $C$. Since $X\cap C$ is countable in $X$, then $g_{|_C}(K\cap(X\cap C))=f_{|_{X\cap C}}(K\cap (X\cap C))$ is open in $Y\cap f(C\cap X)$ as $f_{|_{X\cap C}}$ is a homeomorphism. Thus $K\cap C$ is open in $Y\cap f(X\cap C)$. Also, $g((K^\prime\cap C)\setminus \{k^\prime\})$ is open in $Y^\prime\cap g(C)$ being a set of isolated points. Thus $g_{|_C}$ is an open function. Therefore, $g_{|_C}$ is a homeomorphism.\end{proof}

\bigskip

{\bf Acknowledgment.} I would like to thank professor Arhangel'ski\u{i} for giving us these definitions.

\bigskip

\bigskip

\section*{}
\addcontentsline{toc}{section}{References}


\end{document}